\documentclass[12pt]{article}
\usepackage{hyperref}
\usepackage{graphicx,amsmath,amsfonts,latexsym,amssymb,amsthm,amscd}
\usepackage{latexsym,amscd}

  \newcommand{\beq}{\begin{equation}}
  \newcommand{\eeq}{\end{equation}}

\newcommand{\OM}{\Omega}

\newcommand{\DEL}{\Delta}

\newcommand{\la}{\lambda}

\newcommand{\RR}{\mathbb{R}}
\newcommand{\BR}{\mathbb{R}^{n}}

\newcommand{\lka}{\langle}
\newcommand{\rka}{\rangle}

  \newtheorem{theorem}{Theorem}[section]
\newtheorem{definition}[theorem]{Definition}

  \newtheorem{lemma}[theorem]{Lemma}

\newtheorem{corollary}[theorem]{Corollary}

\begin{document}

\begin{center}
\large{\textbf{Rellich-Kondrakov embedding of the Laplacian resolvent on the torus}}
\end{center}

\begin{center}
 Louis Omenyi \\
Department of Mathematics/Computer Science/Statistics/Informatics, \\
Federal University, Ndufu-Alike, Ikwo,  Nigeria.\\
\qquad\\
Email:  omenyi.louis@funai.edu.ng
\end{center}

\begin{abstract}
 This paper proves  that the domain of the Laplacian, $\DEL ,$  on a  closed Riemannian manifold,  $(M,g),$
   is compactly embedded in $L^{2} (M) .$   Particularly,   the resolvent of the Laplacian, $(\DEL + 1)^{-1},$
    is shown to be  compactly embedded on the torus.
\end{abstract}

\begin{center}
\textbf{Keywords:}   Laplacian;  Resolvent; Sobolev space; Compactly embedding; Riemannian manifold; Torus.
\end{center}
\textbf{MSC class:} 35P20; 35R01.

\section{Introduction}
The  concept of Sobolev spaces are well understood on some  compact and complete Riemannian
manifolds and on  locally compact groups such as the Heisenberg group, see e.g. \cite{6, 8, 9, 10} and  \cite{11}.
In this  paper, we are interested in Sobolev spaces in a general setting 
as a tool for a better understanding of pseudodifferential operators such as the Laplacian on Riemannian manifolds, \cite{1, 7}.
We shall show that the resolvent of the Laplacian is compactly embedded on the torus.

We begin by defining the Sobolev space of integer order  on smooth Riemannian manifolds and proceed to review 
 the Sobolev embedding theorem on  manifolds. Afterwards, we present the so-called Rellich-Kondrakov theorem 
on smooth compact torus which is the main result of the paper.

\section{Materials and Methods}
We gather the concepts of Sobolev space, Sobolev embedding theorems especially for Riemannian manifolds and 
 the Rellich-Kondrakov theorem in this section as tools for the proof of compact embedding  theorems of the
  Laplacian resolvent in the next section. Proofs of some of the theorems may be presented for purpose of completeness.

\subsection{Sobolev Space on  Riemannian manifold}
To begin with, let $\OM$ be an open subset of $\BR$ and $k$ an integer; $p \geq 1$ a real. Let $u : \OM \rightarrow \RR$
 be a real-valued smooth function. Following the works of \cite{1,2,3,4, 14, 13}, let 
\begin{equation}\label{s1}
|| u ||  _{k,p} : = \sum_{0 \leq   | \alpha | \leq  k }  [ \int _{\OM}| D^{\alpha} u| ^{p}  d x ]^{1/p} 
\end{equation}
where the distributional   derivatives 
\[ D^{\alpha} u    = (D^{\alpha _{1}}   \cdots D^{\alpha _{n}} )u ; ~ D_{j}  = \frac{1}{i} \frac{\partial }{\partial x_{j}}  ;\]
  and $\alpha$ is a multi-index. We define the following Sobolev spaces.
\begin{itemize}
\item[1.]   $H_{k} ^{p}$ is the completion of $\{ u \in  C^{\infty} ( \OM ) : || u ||_{k , p}  < \infty \}$ and  
\item[2.]   $W_{k} ^{p}  = \{ u \in L^{p} ( \OM ) :  D^{\alpha} u  \in L^{p} ( \OM ) ; | \alpha | \leq k \}.  $
\end{itemize}

\begin{theorem} \cite{1}. \label{ts1}
For any $\OM ,$ $k$ integer and $p \geq 1;$ $H_{k} ^{p} = W_{k} ^{p} .$
\end{theorem}

Now let $(M,g)$ be a smooth Riemannian $n$-dimensional manifold. For an integer $k$ and 
 \[ u : M \rightarrow \RR \]
  smooth, we denote by $\nabla ^{k} u$ the $k^{th}$ covariant derivative of $u$ and by $| \nabla ^{k} u |$ its norm. 
Specifically,
$$ \nabla ^{0} u = u ; $$
 $$  ~\nabla ^{1} u = ( \nabla  u )_{i} = \partial _{i} u \equiv \frac{\partial u}{\partial x_{i}} ; $$
$$  \nabla ^{2} u = ( \nabla ^{2} u )_{i j}= \partial _{ij} u - \Gamma _{ij}^{k}  \partial _{k} u = \frac{1}{\sqrt{|g|}} \sum _{i,j =1} ^{n} \partial _{i} ( \sqrt{|g|}) g^{ij} \partial _{j} u ; $$
$$ \nabla ^{3} u = ( \nabla ^{2} \nabla u )_{i j k}= \partial _{i} [ \partial _{jk}  - \Gamma _{jk}^{l}  \partial _{l} ] u$$
and so on; where  we recognize $ \nabla ^{2} = \DEL $  as the Laplacian on the manifold $(M,g)$    and
  \[\Gamma _{ij}^{k} : = \frac{1}{2} g^{mk} [ \partial _{i} g _{mj}  + \partial _{j} g _{mi} - \partial _{m} g _{ij} ] \]
 is the Christoffel symbol, see e.g \cite{5, 4}.
Suppose for now that $M = \BR $ with the Euclidean metric, 
we have
\[  \nabla ^{2} u = ( \nabla ^{2} u )_{i j}  = \sum _{i = 1} ^{n} \partial _{i} ^{2} u . \]
Besides, in local coordinates, the norm of $\nabla ^{k} u $ is expressed as 
\[ \| \nabla ^{k} u \| = g^{i_{1} j_{1}  \cdots  i_{k} j_{k} } ( \nabla ^{k} u )_{i_{1}   \cdots  i_{k} }   ( \nabla ^{k} u )_{j_{1}   \cdots  j_{k} } . \]

To define Sobolev space on $(M,g),$ we  set 
\begin{equation}\label{s4}
A_{k} ^{p} (M) = \{ u \in C^{\infty} (M) :  \int _{M} | \nabla ^{j} u | ^{p} dv (g) < \infty ; j = 0, \cdots , k \}.
\end{equation}
For $u \in A_{k} ^{p} (M) , $ we have 
\begin{equation}\label{s5}
|| u || _{H_{k} ^{p}}= \sum _{j =0}  ^{k} [ \int _{M} | \nabla ^{j} u | ^{p}  dv (g) ]^{1/p} 
\end{equation}
where $dv (g) = \sqrt{\det (g_{ij} )} dx $ and the Lebesgue volume measure of $\BR$   is  $dx$.

\begin{definition}
The Sobolev space  $H_{k} ^{p} (M)$ is the completion of  $A_{k} ^{p}(M)$ with respect to $|| \cdot || _{H_{k} ^{p}}$ defined in (\ref{s5}).
\end{definition}

One can look at  $ H_{k} ^{p}(M)$ as a subspace of $L^{p} (M) $ where for $u \in L^{p} (M) ,$ we write 
\[ || u || _{p} =  ( \int _{M} |u|^{p} )^{1/p} .\] 
 Also observe that when $p = 2 ,$ then
\[  || u || _{H_{k} ^{2}}= \sum _{j =0}  ^{k} [ \int _{M} | \nabla ^{j} u | ^{2}  dv (g) ]^{1/2} \]
with the associated inner product $\lka u , v \rka = \sum _{j = 0} ^{k} \int _{M}  \lka  \nabla ^{j} u , \nabla ^{j} v \rka  dv (g) .$

\begin{definition}\cite{4}. 
A real-valued function  $u$ on $M$ is called a Lipschitz function (or Lipschitzian)  if there exists a constant $c > 0$ such that for 
$x , y \in M,$ $|u(y) - u(x)| \leq c~ d_{g} (x,y).$
\end{definition}

We now  look at  smooth Riemannian manifolds. The next theorem is very essential in this regard.
\begin{theorem}\label{p1}\cite{3}
Let $(M,g)$ be a smooth Riemannian $n$-dimensional manifold and let \\ 
$u : M \rightarrow \RR $ be a Lipschitz function on $M$ with compact support.  Then, $u \in H_{1} ^{p} (M)$ for any $p \geq 1.$ 
 In particular, if $M$ is compact, any Lipschitz function in $M$ is also in $H_{1} ^{p} (M) $  too.
\end{theorem}

The property of Sobolev embedding on smooth Riemannian $n$-dimensional manifold is summarised in the next theorem. 
We follow Aubin \cite{3} and Hebey \cite{2} to present a proof of the next theorem here for a purpose of completion.
We denote the space of test functions on a Riemannian manifold $M$  by $ \mathcal{D} (M)  := C_{0}^{\infty}(M)$  and its dual space by 
$ \mathcal{D}' (M).$

\begin{theorem}\label{t2}
Given   that  $(M,g)$  is a smooth  $n$-dimensional complete Riemannian  manifold,
 then $ \mathcal{D} (M) $ is dense in $ H_{1}^{p} (M)$ for any $p \geq 1 .$
\end{theorem}

\begin{proof}
Let $ f : \RR \rightarrow \RR $ be defined by     
\begin{equation*}
f(t) = 
\begin{cases}
1 & \text{~if}~ t \leq 0 \\
1 - t& \text{~if} ~0 \leq t \leq 1\\
0 & \text{~if} ~ t \geq 1 . \\
\end{cases}
\end{equation*}
Let $u \in A_{1} ^{p} (M); ~ p \geq 1$ real. For $x , y \in M ,$ set $u_{j} (y) = u (y) f \big( d_{g} (x, y) - j \big)$
  where $d_{g}$ is the distance associated to $g$ and $j \in \mathbb{Z}.$
By theorem \ref{p1} above, $u_{j} \in H_{1}^{p}(M)$ for any $j .$ Now, since $u_{j} =0 $
 outside any compact set $\OM \subset M ,$ we have that for any $j;$ $u_{j}$ is the limit in  $ H_{1}^{p}(M)$ 
 of some sequence of functions  in  $\mathcal{D}(M) .$  So, if
\[ (u_{m} ) \in  A_{1} ^{p} (M) \rightarrow u_{j} \in H_{1}^{p}(M) \] 
 and if $\alpha \in \mathcal{D}(M), $ then, 
 \[ (\alpha u_{m} ) \in A_{1} ^{p} (M) \rightarrow \alpha u_{j} \in H_{1}^{p}(M) . \]
Consequently, we choose $\alpha \in \mathcal{D}(M) $ such that $\alpha = 1$ where $u_{j} \neq 0.$  
So independently, we have for any $j,$ 
\[  \big( \int _{M} | u_{j} - u |^{p} d v (g) \big) ^{1/p} \leq   \big( \int _{M \setminus B_{r}(j)} | u |^{p} d v (g)  \big)^{1/p}  \]
and
\[    \big( \int _{M} | \nabla (u_{j} - u )|^{p} d v (g)   \big) ^{1/p} \leq  \big( \int _{M \setminus B_{r}(j)} | \nabla u |^{p} d v (g)  \big)^{1/p}
  +  \big( \int _{M \setminus B_{r}(j)} |  u |^{p} d v (g)  \big) ^{1/p} . \]
Hence, $(u_{j}) \rightarrow u \in H_{1}^{p}(M) $ as $j \rightarrow \infty .$  
That is to say, $u$ is the limit in $ H_{1}^{p}(M) $ of some sequence in $ \mathcal{D}(M) .$
\end{proof}

\subsection{Sobolev Embedding}
Given two normed vector spaces $(E, || \cdot ||_{E} ) $ and  $(F, || \cdot ||_{F} ) $ with $E \subset F,$ we say 
that $E$ is continuously embedded in $F$ denoted by $E \hookrightarrow F, $  if   there exist some constant $c > 0$ such that 
for any $x \in E ,$ we have $|| x ||_{F}  \leq  c || x || _{E} .$
We demonstrate this concept for compact smooth Riemannian $n$-dimensional manifolds. 
First, note that the embedding is said to be compact if bounded subsets of  $(E, || \cdot ||_{E} ) $ are pre-compact (relatively compact)
 in $(F, || \cdot ||_{F} ) .$ Clearly, if the embedding $E \hookrightarrow F $ is compact, it is  continuous; see e.g.  \cite{1, 3, 12, 2, 5}.
\begin{theorem}\label{t3}
Let $(M,g)$ be a compact smooth Riemannian $n$-dimensional manifold. For any real numbers $1 \leq q < p$ and integers $0\leq m < k$  satisfying   $\frac{1}{p} = \frac{1}{q} - \frac{k -m}{n} ,$ 
we have $$ H_{k}^{q} (M) \hookrightarrow H_{m}^{p} (M) .$$
\end{theorem}
\begin{proof}
It is enough to prove that $H_{1} ^{1} (M) \hookrightarrow L^{n/(n-1)} $ is valid since for 
 \[ p > 1, ~ H_{k}^{q} (M) \hookrightarrow H_{1}^{1} (M) . \]
Now since $M$ is compact, it can be covered by a finite number of charts $ ( \mathcal{U} _{m} , \psi _{m}) _{m = 1, \cdots , N}$ 
such that for any $m,$  the components  $g_{ij}$  of  $g$  in   $ ( \mathcal{U} _{m} , \psi _{m}) $ 
satisfy 
\[ \frac{1}{2}  \delta _{ij} \leq g_{ij} \leq 2 \delta _{ij}\] 
 as bilinear forms where $ \delta _{ij}$ is the usual metric on $\BR .$  The constant  $c = 2$  can be  chosen for convenience.  

Again, let $( \eta _{m} )$ be a smooth partition of unity subordinate to $( \mathcal{U} _{m} ).$ For any $u \in C^{\infty} (M) ,$ we have 
\[  \int _{M} |\eta _{m} u  |^{n/(n-1)} d v (g) \leq 2^{n/2} \int _{\BR} |( \eta _{m} u ) \circ \psi _{m} ^{- 1} (x) |^{n/(n-1)} dx    \]
and
\[  \int _{M} |\nabla (\eta _{m} u  )|^{n/(n-1)} d v (g) \geq 2^{(n+1)/2} \int _{\BR} |\nabla (( \eta _{m} u ) \circ \psi _{m} ^{- 1} )(x) |dx  .   \]
Hence,
\[  \big( \int _{\BR} |( \eta _{m} u ) \circ \psi _{m} ^{- 1} (x) |^{n/(n-1)} dx \big) ^{(n-1)/n} 
  \leq  \frac{1}{2} \int _{\BR} |\nabla (( \eta _{m} u ) \circ \psi _{m} ^{- 1} )(x) |dx \]
for any $m .$  Consequently, for any $ u \in C^{\infty} (M) ,$
\[ \big( \int _{M} | u  |^{n/(n-1)} d v (g)  \big) ^{(n- 1)/n}   \leq   \sum _{m = 1} ^{N} \big( \int _{M} | \eta _{m} u | ^{n/(n-1)}
 d v (g) \big)^{(n -1 )/n}  \]
\[  \leq  2^{n-1 } \sum _{m = 1} ^{N} \int _{M} | \nabla ( \eta _{m} u ) | d v (g)  
  \leq  2^{n-1 } \int _{M} | \nabla u  | d v (g) + 2^{n -1} \big( \max _{M} \sum _{m = 1}^{N} | \nabla \eta _{m} | \big) \int _{M} |u| dv(g) \]
which gives the result.
\end{proof}

\subsection{Rellich-Kondrakov theorem}
Now, we are set to present the Rellich-Kondrakov theorem (which may also be called Rellich theorem for simplicity).
\begin{theorem}\label{t4}
Let $(M,g)$ be a compact smooth Riemannian $n$-dimensional manifold. For any integers $j \geq 0;~ m \geq 1 ;$
 and real numbers $q \geq 1,~ p$ such that  $1 \leq p \leq \frac{nq}{n - mq} ,$ the embedding 
\[ H_{j + m} ^{q} (M) \hookrightarrow H_{j} ^{p} (M) \]
  is compact. 
In particular, for any $q \in [ 1, n )$ real and any $p \geq 1$ such that    $ \frac{1}{p} > \frac{1}{q} - \frac{1}{n} ,$ the embedding $H_{1}^{q}(M) \hookrightarrow L^{p}(M)$ is compact.
\end{theorem}
To prove the  Rellich theorem, we need the following lemma. 
\begin{lemma}\label{l6}
Let $\OM $ be a bounded open subset of $\BR ; ~ q \in [1 ,n )$  real such that $\frac{1}{p} > \frac{1}{q} - \frac{1}{n} .$ Then, the embedding $H_{0,1}^{q}(\OM) \hookrightarrow L^{p}(\OM)$ is compact; where 
$H_{0,1} (\OM )$ denotes the closure of $\mathcal{D} (\OM ) \subset H_{1}^{q}(\OM ) .$
\end{lemma}
For proof, one can see Aubin  \cite{3}. Next is the proof of theorem \ref{t4}.
\begin{proof}
Since $M$ is compact, it can be covered by a finite number of charts $( \mathcal{U} _{s}, \psi _{s} )_{s = 1, \cdots , N}$ such that for any $s,$ the components $g_{ij} ^{s}$ of $g$ in  $( \mathcal{U} _{s}, \psi _{s} ) $ 
satisfy  $\frac{1}{2} \delta _{ij} \leq g_{ij} ^{s} \leq 2 \delta _{ij} $ as bilinear forms.
Let $ (\eta _{s}) $ be a smooth partition of unity subordinate to the covering  $( \mathcal{U} _{s} ) .$ Given $ ( f_{m} )$ a bounded sequence in $  H_{1} ^{q} (M)  $ and for any $s ,$ we let $f_{m} ^{s} = ( \eta _{s} f_{m}) \circ \psi _{s}^{- 1}.$
Clearly, $(f_{m} ^{s} )$ is a bounded sequence in  $H_{0,1}^{q}(\mathcal{U} _{s}, \psi _{s} )) $ for any $s .$

By lemma \ref{l6}, one then gets that a subsequence $(\tilde{f}_{m} ^{s} )$ of $(f_{m} ^{s} )$ is a Cauchy sequence in $L^{P} ( \psi _{s}( \mathcal{U} _{s} ) ) .$ 

Coming back to the inequality satisfied by the   $ g_{ij} ^{s} , $ one gets for any $s,$   $(\eta  _{s}  f_ {m} ) $ is a Cauchy sequence in $L^{p} (M).$ But for any $m_ {1}, m_{2} ,$
$$ || f_{m_{1}} - f_{m_{2}} ||_{p} \leq  \sum _{s =1} ^{N} ||\eta _{s} f_{m_{1}} - \eta _{s} f_{m_{2}} ||_{p}; $$
where $|| \cdot ||_{p}$  denotes the $L^{p}$-norm. Hence, $( f_{m} )$ is a Cauchy sequence in $L^{p} (M).$ This proves the theorem.
\end{proof}

\begin{corollary}
The embedding $H_{0}^{1} (M) \cap H_{1}^{2} (M)  \hookrightarrow H_{1}^{1} (M)  \hookrightarrow L^{2} (M) $ is compact.
\end{corollary}
\begin{proof}
Follows from  the  Rellich-Kondrakov theorem proved in lemma \ref{l6}.
\end{proof}

\section{Results and Discussion}
The compactness of  $ (\DEL + 1)^{ -1} $  on  $( M, g ) $  is more conveniently  discussed on Fourier space, \cite{12, 13}  and  \cite{15}.
So we continue the discussion on  the Sobolev space. The Sobolev space $H^{s} (\BR )$  is  defined by
\[  H^{s} (\BR ) = \{ u \in \mathcal{S}^{\prime} (\BR ) : (1 + | \xi |^{2})^{s}| \hat{u} |^{2} \in L^{2}(\BR) \} \]
where $\hat{u}$ is the Fourier transform of $u .$
This means that  a function $u \in   H^{s} (\BR ) $ provided 
$\int _{\BR }   ( 1 + | \xi |^{2} )^{s}      | \hat{u} ( \xi ) |^{2}  d \xi < \infty $ with 
$$|| u ||^{2} _{H^{s}} = \int _{\BR }   ( 1 + | \xi |^{2} )^{s}      | \hat{u} ( \xi ) |^{2}  d \xi .$$

Now, define $ \mathbb{T} ^{n}  = \BR /2 \pi \mathbb{Z}^{n} = ( S^{1} ) ^{n}$ to be the  $n$-dimensional unit torus.
 A function  $u \in \mathbb{T} ^{n }$  can be expressed via Fourier series as
 \[ u (x) = \sum _{\xi \in \mathbb{Z}^{n} } \hat{u} ( \xi ) e^{i \xi \cdot x}  
~~ \mathrm{where} ~~ \hat{u} = ( 2 \pi ) ^{- n } \int _{\mathbb{T} ^{n}} u(x) e^{- i \xi \cdot x}  dx . \]
For any function $u \in \mathcal{D} ^{\prime} ( \mathbb{T} ^{n} ) ,$  we can write  
\[ \hat{u} = ( 2 \pi ) ^{- n }  \lka  u(x) , e^{- i \xi \cdot x}  \rka \]
  so that by Plancherel theorem
\[ \sum _{\xi \in \mathbb{Z}^{n} } | \hat{u} ( \xi ) |^{2} =   ( 2 \pi ) ^{- n } \int _{\mathbb{T} ^{n}} | u(x)|^{2}   dx \]
 and
  \[ D^{\alpha} u(x) = \sum _{\xi \in \mathbb{Z}^{n} }  \xi ^{\alpha } \hat{u} ( \xi ) e^{i \xi \cdot x} . \]
So, it is now clear that $u \in C^{\infty} ( \mathbb{T} ^{n} )$ if and only if  $  \hat{u} $ is a rapidly decreasing function in 
 $\mathbb{Z}^{n} .$  That is, for $s \in \RR ,$
\[ \sup _{\xi \in \mathbb{Z}^{n} }   ( 1 + | \xi |^{2} )^{s}      | \hat{u} ( \xi ) |^{2}  d \xi < \infty . \]
  By duality,  $u \in \mathcal{D} ^{\prime} ( \mathbb{T} ^{n} ) $  provided $\hat{u} $ is polynomially bounded function, i.e.
\[ | \hat{u} |  =  c ( 1 + | \xi |^{2} )^{l} \] 
 for  $c > 0, ~ l \in \RR .$

Let   $T$ defined by  $T = ( \DEL +1 )^{- 1}$  be the Laplacian resolvent on $L^{2} (\mathbb{T} ^{n} ) $ where we know that 
\begin{eqnarray*}
\mathcal{F} : L^{2} (\mathbb{T} ^{n} ) \rightarrow l^{2} (\mathbb{T} ^{n} ) ~ \mathrm{and} ~ \\
 T = \mathcal{F} ^{ - 1}  ( 1 + | \xi |^{2} )^{-1} \mathcal{F} ; \\
   \Rightarrow || T || = 1.  
\end{eqnarray*}

The next theorem is the main result of this work.
\begin{theorem}
T is a compact operator on   $ l^{2} (\mathbb{T} ^{n} ) .$
\end{theorem}

\begin{proof}
It enough to show that $T$ is the limit of  finite rank operators. To do this, define
\begin{equation*}
 ( T_{N} u ) ( \xi ) = 
\begin{cases}
( 1 + | \xi |^{2} )^{-1} u_{\xi} & ; ~ | \xi |  \leq  N \\
0 & ; ~ | \xi |  >  N . \\
\end{cases}
\end{equation*} 
for $N > 0 $ and $ u_{\xi}  \in l^{2} (\mathbb{T} ^{n} ). $ Then,
\begin{equation*}
 (T -  T_{N} ) \psi _{\xi} = 
\begin{cases}
( 1 + | \xi |^{2} )^{-1} \psi _{\xi}  & ; ~ | \xi |  \geq  N + 1\\
0 & ; ~ | \xi |  \leq   N . \\
\end{cases}
\end{equation*}
So, 
$$ ||T - T_{N} || = \frac{1}{( N + 1 )^{2} + 1} .$$
But, 
\begin{eqnarray*}
     (T - T_{N}  ) \psi _{\xi}  & \leq  & \frac{1}{( N + 1 )^{2} + 1}  || \psi _{\xi}||  \\
    \Rightarrow   \big( \sum _{\xi \in  \mathbb{Z}^{n} }  | ( 1 + | \xi |^{2} )^{-1} \psi _{\xi} |^{p} \big)^{1/p} 
& \leq & \frac{1}{ (N^{2} +1 )^{p}}   \sum _{\xi \in  \mathbb{Z}^{n} }   | \psi _{\xi} |^{1/p} \\
 & =&   \frac{1}{ N^{2} +1}  \sum _{ |\xi  | \geq N + 1 }   | \psi _{\xi} |^{p}  \\
 &  \leq &  \frac{1}{ N +1}  || \psi _{\xi} ||_{l^{p}} \rightarrow 0  ~~~\mathrm{as~~ N ~~\rightarrow ~~\infty .} 
\end{eqnarray*}
  Therefore, $T = ( \DEL + 1 ) ^{-1}$ is compact in   $ l^{2} (\mathbb{T} ^{n} ) . $
\end{proof}

Another result  of the Rellich-Kondrakov theorem for the torus is the next theorem.
\begin{theorem}
Suppose $T$ is a self-adjoint operator (or has self-adjoint extension) with compact resolvent, 
then $T$ has discrete spectrum.
\end{theorem}

\begin{proof}
It is straightforward to see that 
\begin{eqnarray*}
 (T - \mathrm{id})^{-1} \psi _{k} &=& \mu _{k} \psi _{k} \\
       \Rightarrow  \psi _{k} &=& \mu _{k} (T - \mathrm{id}) \psi _{k} \\
  \Rightarrow  \psi _{k}  +   \mu _{k} \psi _{k}  &=& \mu _{k} T \psi _{k} \\
  \Rightarrow T \psi _{k} &=& \frac{ ( \mathrm{id} + \mu _{k})     }{\mu _{k}} \psi _{k} .
\end{eqnarray*}
Set 
\[ \la _{k} = \frac{ ( \mathrm{id} + \mu _{k})     }{\mu _{k}} \] 
 and write 
\[ T \psi _{k}  =  \la _{k} \psi _{k} . \]
We see that as $ \mu _{k} \rightarrow 0 , $ $ \la _{k} \rightarrow \infty $ which confirms the compactness
of the resolvent $(T - \mathrm{id})^{-1}$ and that $T$ has discrete spectrum 
$\{  \la _{k} = \frac{ ( \mathrm{id} + \mu _{k})     }{\mu _{k}}  \} .$ 
\end{proof}

\section{Conclusion}
The paper has given a  proof  that the Laplacian resolvent operator is compact on the unit torus by means
of the Fourier transform. We further found that given a compact Riemannian manifold,  $(M,g)$,  that 
the embedding $H_{0}^{1} (M) \cap H_{1}^{2} (M)  \hookrightarrow H_{1}^{1} (M)  \hookrightarrow L^{2} (M) $ is compact.

For the Laplacian resolvent, $ (\DEL + 1)^{ -1} ,$  we consequently showed  that as a compact operator on the torus, it has discrete spectrum.
Thus, $ (\DEL + 1)^{ -1} $ satisfies the Rellich-Kondrakov theorem  on  $ \mathbb{T} ^{n} .$

\section{Conflict of Interest}
There is no conflict of interest.

\end{document}